\documentclass[11pt]{amsart}
\usepackage{amsmath,amssymb,amsthm}
\usepackage{graphicx}
\usepackage{tikz}
\theoremstyle{plain}
\newtheorem{theorem}{Theorem}%[section]
\newtheorem{lemma}{Lema}%[section]
\newtheorem{example}{Example}%[section]
\newtheorem{proposition}{Proposition}%[section]
\newtheorem{corollary}{Corollary}%[section]
\newtheorem*{definition}{Definition}%
\newtheorem{remark}{Remark}%[section] 

\begin{document}

\title{Three-dimensional manifolds with poor spines}

\author{E.~A.~Fominykh, V.~G.~Turaev,  A.~Yu.~Vesnin}

\address{Fominykh: Chelyabinsk State University, Chelyabinsk, Russia; Krasovskii Institute of Mathematics and Mechanics, Yekaterinburg, Russia.} 

\email{fominykh@csu.ru}

\address{Turaev: Chelyabinsk State University, Chelyabinsk, Russia; Indiana University, Bloomington, IN, USA.}

\email{vturaev@yahoo.com}

\address{Vesnin: Chelyabinsk State University, Chelyabinsk, Russia; Sobolev Institute of Mathematics, Novosibirsk, Russia.}

\email{vesnin@math.nsc.ru}

\thanks{The work was supported by the Laboratory of Quantum Topology, Chelyabinsk State University (contract  14.Z50.31.0020 with the Ministry of Education and Sciences of the Russian Federation), by the Russian Foundation for Basic Research (projects 14-01-00441 (Fominykh) and 13-01-00513 (Vesnin), and by a grant of the President of the Russian Federation (project NSh-1015.2014.1).}

\begin{abstract}
A special spine of a three-manifold is said to be \emph{poor} if it does not contain proper simple subpolyhedra. Using the Turaev--Viro invariants, we establish that every compact three-dimensional manifold $M$ with connected nonempty boundary has a finite number of poor special spines. Moreover, all poor special spines of the manifold $M$ have the same number of true vertices. We prove that the complexity of a compact hyperbolic three-dimensional manifold with totally geodesic boundary that has a poor special spine with two 2-components and $n$ true vertices is $n$. Such manifolds are constructed for infinitely many values of $n$. 
\end{abstract}

%%% ----------------------------------------------------------------------

\maketitle

\section{Introduction}

The notion of complexity of a 3-manifold was introduced by S.~Matveev and plays an important role in the theory of 3-manifolds \cite{1}. The tabulation of 3-manifolds of given complexity and the calculation of the exact values of complexity for large classes of manifolds provide a natural approach to the manifold classification problem. The problem of calculating the complexity of manifolds is quite difficult. At present, the exact values of complexity are known only for a finite number of tabulated manifolds \cite{2, 3, 4} and for several infinite families of manifolds with boundary \cite{5, 6, 7, 8}, closed manifolds \cite{9, 10}, and cusped manifolds \cite{11, 12}. Estimates for the complexity of some infinite families of manifolds were obtained in \cite{13, 14, 15, 16, 17}.

In this paper, we prove that if a hyperbolic 3-manifold with totally geodesic boundary has a poor special spine with two 2-components and $n$ true vertices, then its complexity is $n$ (Theorem~4). We construct examples of such manifolds for infinitely many values of $n$. These examples give a new infinite family of 3-manifolds for which the exact values of complexity are known.

The paper is organized as follows. In Section~2, we present the basic notions of the complexity theory of 3-manifolds. In Section~3, we introduce the notion of a poor spine, give examples of hyperbolic 3-manifolds with totally geodesic boundary that have poor special spines, and study the existence of poor spines in manifolds. In Section~4, we establish that every compact 3-manifold with nonempty boundary has a finite number of poor special spines and that these spines have the same number of true vertices (Proposition~4). In Section~5, we introduce classes ${\mathcal M}^{k}_{n}$ of hyperbolic manifolds with totally geodesic boundary that have poor special spines with the number of 2-components equal to $k$ and the number of true vertices equal to~$n$. We prove a theorem on the complexity of manifolds in the classes ${\mathcal M}^{2}_{n}$. In Section~6, we construct an infinite family of such manifolds.

\section{Polyhedra and spines of 3-manifolds} \label{sec2}

Recall the definitions of simple and special polyhedra (we follow the book~\cite{1}).

\begin{definition} {\rm 
A compact polyhedron $P$ is said to be \emph{simple} if the link of each of its points $x \in P$  is homeomorphic to one of the following one-dimensional polyhedra:
\begin{itemize}
\item[(a)] a circle (in this case the point $x$ is said to be \emph{nonsingular}); 
\item[(b)] a union of a circle and a diameter (in this case the point $x$ is called a \emph{triple point});
\item[(c)] a union of a circle and three radii (in this case the point $x$ is called a \emph{true vertex}).
\end{itemize}
}
\end{definition}

We will call the points of types (b) and (c) \emph{singular} points. The set of singular points of a simple polyhedron $P$ is called its \emph{singular graph} and is denoted by $SP$. In the general case, the set $SP$ is not a graph on the set of true vertices of the polyhedron $P$, because $SP$ may contain closed triple lines without true vertices. If there are no closed triple lines, then $SP$ is a $4$-regular graph (each vertex is incident to exactly four edges) that may have loops and multiple edges.

Every simple polyhedron has a natural \emph{stratification}. The strata of dimension 2 (\emph{2-components}) are the connected components of the set of nonsingular points. The strata of dimension $1$ are the open or closed triple lines. The strata of dimension $0$ are the true vertices. It is natural to require that each stratum be a cell, i.e., that a polyhedron $P$ be a cell complex.

\begin{definition} {\rm 
A simple polyhedron $P$ is said to be \emph{special} if the following conditions are satisfied: 
\begin{itemize}
\item[(1)] each one-dimensional stratum of the polyhedron $P$ is an open $1$-cell;
\item[(2)] each $2$-component of the polyhedron $P$ is an open $2$-cell.
\end{itemize}
}
\end{definition} 

It is obvious that if the polyhedron $P$ is connected and contains at least one true vertex, then condition (2) implies condition (1).

\begin{definition}{\rm 
Let $M$ be a connected compact $3$-manifold. A compact two-dimensional polyhedron $P \subset M$ is called a \emph{spine} of $M$ if either $\partial M \neq \emptyset$ and $M \setminus P$ is homeomorphic to $\partial M \times (0,1]$ or  $\partial M = \emptyset$ and $M \setminus P$ is homeomorphic to an open ball. A \emph{spine} of a disconnected $3$-manifold is the union of spines of its connected components.
}
\end{definition}

\begin{definition}{\rm 
A spine of a $3$-manifold is said to be \emph{simple} or \emph{special} if it is a~simple or special polyhedron, respectively.
} 
\end{definition}

The importance of studying special spines of $3$-manifolds is associated with the following facts.

\begin{theorem} \cite{18} 
Any compact $3$-manifold has a special spine.
\end{theorem}

\begin{theorem} \cite{1} 
If two compact connected $3$-manifolds have homeomorphic special spines and both manifolds are either closed or have nonempty boundaries, then these manifolds are homeomorphic.
\end{theorem}

Note that not any special polyhedron defines a $3$-manifold. Namely, there exist unthickenable polyhedra that cannot be embedded in $3$-manifolds \cite{1}.

A compact polyhedron $P$ is said to be \emph{almost simple} if the link of each of its points can be embedded in a complete graph $K_{4}$ on four vertices. The points whose links are homeomorphic to the graph $K_{4}$ are true vertices of the polyhedron $P$. A spine of a manifold is said to be \emph{almost simple} if it is an almost simple polyhedron. We say that the complexity $c(M)$ of a manifold M is equal to $n$ if $M$ has an almost simple spine with $n$ true vertices and has no almost simple spines with a smaller number of true vertices.

\begin{theorem} \cite{1} 
Suppose that a compact irreducible boundary-irreducible $3$-manifold $M$ is such that $M \neq D^{3}, S^{3}, {\mathbb R}P^{3}, L_{3,1}$  and all proper annuli in $M$  are inessential. Then, for any almost simple spine $P$ of $M$, one can find a special spine $P_{1}$ of $M$ with a smaller or the same number of true vertices.
\end{theorem} 

In particular, this theorem applies to compact hyperbolic $3$-manifolds. Thus, for such manifolds, the spines on which the minimum number of true vertices is attained can be sought in the class of special spines.

Now, following~\cite{6}, we describe the relation between special spines and triangulations of compact manifolds. An \emph{ideal tetrahedron} is a tetrahedron with removed vertices. By an \emph{ideal triangulation}  of a compact manifold $M$ with boundary one means a representation of the interior of $M$ as a set of ideal tetrahedra glued together via pairwise identifications of their faces. Each ideal triangulation of a manifold $M$ naturally defines a dual special spine. Namely, suppose that the manifold $M$ is obtained by gluing ideal tetrahedra $T_{1}, \ldots, T_{d}$. In each tetrahedron $T_{i}$, consider a union $R_{i}$ of the links of all four (removed) vertices of the tetrahedron in its first barycentric subdivision (see Fig.~1). To the gluing of the tetrahedra $T_{1}, \ldots, T_{d}$, one assigns the gluing of the polyhedra $R_{1}, \ldots, R_{d}$.  As a result, one obtains a special spine $P = \cup_{i=1}^{d} R_{i}$ of the manifold $M$. The above assignment induces a bijection between ideal triangulations (considered up to equivalence) and special spines (considered up to homeomorphisms).

\begin{figure}[h]
\centering
\includegraphics[scale=0.7]{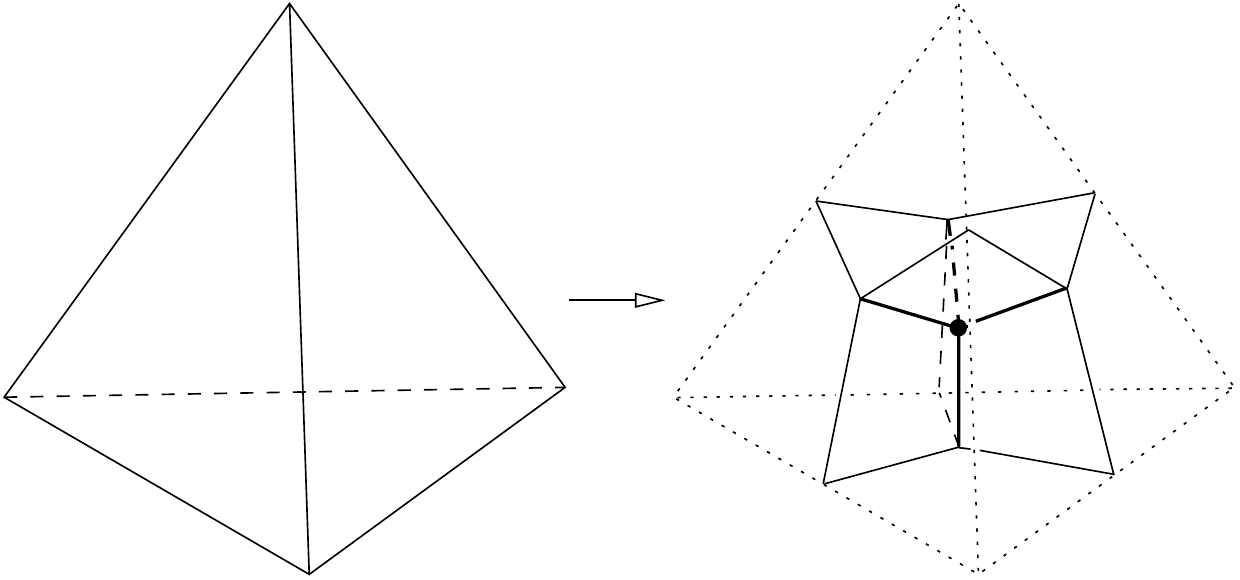}
\caption{Tetrahedron $T_{i}$ and polyhedron $R_{i}$.} \label{fig: dualspine}
\end{figure}

Recall the notion of a hyperbolic truncated tetrahedron~\cite{6}. Let $T$ be a tetrahedron and $T^{*}$ be a combinatorial polyhedron obtained from $T$ by removing sufficiently small open stars of its vertices. A \emph{lateral hexagon} and a \emph{truncation triangle} are the intersection of $T^{*}$ with a face and a link of a vertex of $T$, respectively. The edges of truncation triangles are called \emph{boundary edges}, and the other edges of $T^{*}$ are called \emph{internal edges}. A \emph{hyperbolic truncated tetrahedron} is the realization of $T^{*}$ as a compact polyhedron in $\mathbb H^{3}$ such that the truncation triangles are geodesic triangles, the lateral hexagons are geodesic hexagons, and the truncation triangles and lateral hexagons lie in the planes making right angles with each other. A truncated tetrahedron is said to be \emph{regular} if all its dihedral angles at internal edges are equal. For every $\theta$ such that $0 < \theta < \pi/3$, there exists a unique regular truncated tetrahedron $T^{*}_{\theta}$ with dihedral angle $\theta$. The lengths of all lateral edges in $T^{*}_{\theta}$ are equal, just as the lengths of all internal edges.

Formulas for the volume of a regular truncated hyperbolic tetrahedron were obtained in~\cite{19}. The first of them is
\begin{equation} \label{eq1} 
\operatorname{vol} (T^{*}_{\theta}) =  8 \Lambda \left(\frac{\pi}{4} \right) - 3 \int_{0}^{\theta} \operatorname{arccosh} \left( \frac{\cos t}{2 \cos t - 1} \right) \, dt ,
\end{equation}
where $\Lambda (x) = - \int_{0}^{x} \ln | 2 \sin \zeta | d \zeta$ is the Lobachevsky function. The second formula express the volume in terms of values of the Lobachevsky function: 
\begin{equation} \label{eq2} 
\begin{gathered}
\operatorname{vol} (T^{*}_{\theta}) = 6 \left[ \Lambda \left( \frac{\pi}{3} + \varphi \right) - \Lambda \left( \frac{\pi}{3} - \varphi \right) + \Lambda \left( \frac{5\pi}{6} - \varphi \right) + \Lambda \left( \frac{\pi}{6} - \varphi \right) \right. \\ 
\left. + \Lambda \left( \frac{\theta}{2} + \varphi \right) - \Lambda \left( \frac{\theta}{2} - \varphi \right) + 2 \Lambda \left( \frac{\pi}{2} - \varphi \right)  \right] ,
\end{gathered}
\end{equation}
where $\varphi = \arctan \frac{\sqrt{1 - 3 \sin^{2} \theta/2}}{\cos \theta/2}$. 

\section{Poor spines} \label{sec3}

Among all simple polyhedra, we distinguish a class of those that do not contain proper simple subpolyhedra.

\begin{definition}
{\rm  A simple polyhedron is said to be \emph{poor} if it does not contain proper simple subpolyhedra.}
\end{definition}

It follows from the definition that a poor polyhedron is always connected. Below we will consider only connected simple polyhedra.

\begin{definition}
{\rm A simple spine of a $3$-manifold is said to be \emph{poor} if it is a poor polyhedron.}
\end{definition}

It is obvious that if a manifold is disconnected, then any of its simple spines is not poor. Below we will consider only connected manifolds. We will focus on the special spines of these manifolds.

\begin{lemma}  \label{lemma1}  
Any special polyhedron with one 2-component is poor.
\end{lemma}

\begin{proof} 
Indeed, suppose that a special polyhedron $P$ has exactly one $2$-component, which we denote by $\xi$. If a simple polyhedron $Q \subset P$ contains at least one point of the $2$-component $\xi$ of $P$, then $\xi \subset Q$ by the compactness of simple polyhedra. Hence, the polyhedron $Q$ is either empty or coincides with $P$. 
\end{proof}

Consider examples of $3$-manifolds that have special spines with one $2$-component. 

\begin{example} \label{example1} 
{\rm  In~\cite{6}, for every $n \geq 2$, the authors defined and studied a class $\mathcal M_{n}$. An orientable $3$-manifold belongs to the class $\mathcal M_{n}$ if it has a special spine with $n$ true vertices and exactly one $2$-component. Manifolds of the class $\mathcal M_{n}$ possess the following characteristic properties. If $M \in \mathcal M_{n}$,  then the manifold $M$ is hyperbolic with totally geodesic boundary of genus $n$, its complexity $c(M)$ is equal to $n$, and its volume is calculated by the formula
\begin{equation} \label{eq3} 
\operatorname{vol} (M) = n \cdot \left[ 8 \Lambda \left(\frac{\pi}{4} \right) - 3 \int_{0}^{\pi/(3n)} \operatorname{arccosh} \left( \frac{\cos t}{2 \cos t - 1} \right) \, dt \right] .
\end{equation} 
The number of manifolds in $\mathcal M_{n}$ grows exponentially as $n \to \infty$.}
\end{example}

\begin{example} \label{example2} 
{\rm M.~Fujii~\cite{20} showed that there exist exactly eight different compact orientable hyperbolic $3$-manifolds  with totally geodesic boundary which are obtainable by gluing together two hyperbolic truncated tetrahedra, and that the boundary of each of them is a surface of genus $2$. He gave explicit descriptions of truncated triangulations of these manifolds. It follows from the formula for the Euler characteristic that the special spines dual to the triangulations have one $2$-component. Thus, Fujii's eight manifolds constitute the class $\mathcal M_{2}$ from Example~\ref{example1}. By Lemma~1, their spines are poor.
}  
\end{example}

\begin{example}  \label{example3} {\rm 
In~\cite{3}, R.~Frigerio, B.~Martelli, and C.~Petronio classified compact orientable hyperbolic $3$-manifolds with totally geodesic boundary that have complexity $3$ and $4$. Truncated triangulations of these manifolds (which are not necessarily minimal) are presented on the homepage of C.~Petronio~\cite{24}. There are $150$  manifolds of complexity $3$. Among them $74$ manifolds have boundary of genus $3$ and belong to the class $\mathcal M_{3}$ described in Example~\ref{example1}. Hence, they have poor special spines. The results of computer analysis carried out by V.V.~Tarkaev have allowed us to find out that among special spines dual to the remaining $76$ truncated triangulations, $45$ spines are poor (see Table~1).

 \begin{table}[h!]
 \begin{center}
\caption{Number of poor spines of manifolds of complexity $3$} 
\label{comp3}
\begin{tabular}{|c|c|c|} \hline
Genus of the boundary & Total number of spines & Number of poor spines \\ \hline
 $3$ &   $74$ & $74$ \\ \hline
 $2$ &  $76$ & $45$ \\ \hline
Total &   $150$ & $119$ \\ \hline
 \end{tabular}
\end{center}
\end{table}%

There are $5002$ manifolds of complexity $4$. Among them, $2340$ manifolds have boundary of genus $4$ and belong to the class $\mathcal M_{4}$ described in Example~\ref{example1}. Hence, they have poor special spines. $2034$ manifolds have boundary of genus $3$, and $628$ manifolds have boundary of genus $2$. The computer analysis carried out by V.V.~Tarkaev has allowed us to find out that among $2662$ special spines dual to the truncated triangulations of these manifolds, $2010$ spines are poor. More detailed data according to the type of the Kojima decomposition are given in Table~2. 

 \begin{table}[h!]
 \begin{center}
\caption{Number of poor spines of manifolds of complexity $4$}
\label{comp4}
\begin{tabular}{|c|c|c|c|}
\hline
Genus   & Kojima decomposition & Total number  & Number  \\
of the boundary & & of spines & of poor spines \\
\hline
 $4$ &  & $2340$ & $2340$ \\ \hline
 $3$ & Four tetrahedra & $1936$ & $1421$ \\ \hline
 $3$ & Five tetrahedra & $42$ & $0$ \\ \hline
 $3$ & $\pi/6$-octahedron  & $56$ & $0$ \\ \hline
 $2$ & Four tetrahedra & $555$ & $525$ \\ \hline
 $2$ & Five tetrahedra & $41$ & $33$ \\ \hline
 $2$ & Six tetrahedra & $3$ & $3$ \\ \hline
 $2$ & Eight tetrahedra & $3$ & $3$ \\ \hline
 $2$ & $\pi/3$-octahedron & $14$ & $13$ \\ \hline
 $2$ & Octahedron  & $8$ & $8$ \\ \hline
 $2$ & Two pyramids & $4$ & $4$ \\ \hline
 Total &  & $5002$ & $4350$ \\ \hline
 \end{tabular}
\end{center}
\end{table}% 
}
\end{example}

Note the following fact obtained from the analysis of Table~2.

\begin{remark} 
A poor special spine of a hyperbolic $3$-manifold may not be minimal with regard to the number of true vertices.
\end{remark} 

The notion of a poor spine introduced above allows us to reformulate Corollary~12 from~\cite{21} as follows.

\begin{proposition}
Among all special spines of closed $3$-manifolds, there exists exactly one poor spine, the minimal special spine of the lens space $L_{5,2}$. 
\end{proposition} 

\begin{proposition}
Let $M$ be a connected compact $3$-manifold whose boundary consists of $k \geq 2$ components. Then no special spine of the manifold $M$ is poor.
\end{proposition}

\begin{proof} 
Let $P$ be a special spine of $M$. Then the space $M \setminus P$ consists of $k \geq 2$ connected components. Denote by $B_{1}$ one of these components and by $B_{2}$ the union of the remaining components. Assign blue color to those $2$-components of the spine $P$ that lie in the intersection of the closure of the space $B_{1}$ with the closure of the space $B_{2}$. Then the union of the closures of blue $2$-components of the spine $P$ is a proper simple subpolyhedron of this spine.
\end{proof}

\begin{corollary}
Let a $3$-manifold $M$ have a poor special spine $P$ different from the minimal special spine of the lens space $L_{5,2}$. Then the boundary of $M$ has exactly one component and this component is different from the sphere.
\end{corollary} 

\section{Finiteness of the number of poor spines of a manifold} \label{sec4}

A particular case of the Turaev--Viro invariants of $3$-manifolds~\cite{22} is the $\varepsilon$--invariant, which is
defined as follows (see~\cite{1}). Let $P$ be a special spine of a compact manifold $M$. Denote by $\mathcal{F}(P)$ the set of all its simple subpolyhedra, including $P$ and the empty set. To each simple polyhedron $Q\subset P$, assign its $\varepsilon$-weight by the formula  
$$
w_{\varepsilon}(Q) = (-1)^{V(Q)}\varepsilon^{\chi(Q)-V(Q)},
$$
where $V(Q)$ is the number of true vertices in $Q$, $\chi(Q)$ is the Euler characteristic of the polyhedron $Q$, and  $\varepsilon = (1+\sqrt{5})/2$ is a solution of the equation $\varepsilon^2=\varepsilon+1$. Then the $\varepsilon$-invariant  $t(M)$ of the manifold $M$ is defined by the formula 
$$
t(M) = \sum_{Q\in \mathcal{F} (P)} w_{\varepsilon}(Q).
$$

It is known that the Euler characteristic $\chi(M)$ of a compact manifold $M$ with nonempty boundary is equal to the Euler characteristic $\chi(P)$ of any of its spines. The definition of the $\varepsilon$-invariant implies the following proposition. 

\begin{proposition} \label{prop-formula}
Let $M$ be a compact $3$-manifold that has a poor special spine with $V$ true vertices. Then
$$
t(M) = (-1)^{V} \varepsilon^{\chi(M) - V} +1 . 
$$
\end{proposition}

\begin{proof}
Let $P$ be a poor special spine of $M$ with $V$ true vertices. Then ${\mathcal F} (P) = \{ \emptyset, P \}$. The assertion follows from the facts that $w_{\varepsilon}(\emptyset) = 1$ and $w_{\varepsilon}(P) = (-1)^{V}\varepsilon^{\chi(M)-V}$.
 \end{proof}
 
\begin{proposition}  \label{prop4} 
Every compact $3$-manifold $M$ with connected nonempty boundary has a finite number of poor special spines. All poor special spines of the manifold $M$ have the same number of true vertices.
 \end{proposition} 
 
\begin{proof} 
Let $P_{1}$ and $P_{2}$ be poor special spines of the manifold $M$ that have $V_{1}$ and $V_{2}$ true vertices, respectively. Let us calculate the value of $t(M)$ by the formula from Proposition~\ref{prop-formula} in two ways, using the spines $P_{1} $ and $P_{2}$. Then   
$$
(-1)^{V_{1}}\varepsilon^{\chi(M)-V_{1}}  =  (-1)^{V_{2}}\varepsilon^{\chi(M)-V_{2}}, 
$$
whence $V_{1} = V_{2}$. Thus, all poor special spines of the manifold $M$ have the same number of true vertices.

The finiteness of the number of poor special spines of the manifold $M$ follows from the theorem on the finiteness of the number of special spines with a fixed number of true vertices~\cite{1}. 
\end{proof}

\section{Complexity of hyperbolic manifolds with poor spines} \label{sec5}

Introduce a  class of manifolds ${\mathcal M}^{k}_{n}$, where $k \geq 1$  and $n \geq 1$.  

\begin{definition}{\rm 
A connected compact orientable hyperbolic $3$-manifold $M$ with nonempty totally
geodesic boundary belongs to the class ${\mathcal M}^{k}_{n}$ if it has a poor special spine with the number of
2-components equal to $k$ and the number of true vertices equal to $n$.
}
\end{definition} 

\begin{proposition} \label{prop5}
If $k \geq n$, then ${\mathcal M}^{k}_{n} = \emptyset$. 
\end{proposition}

\begin{proof}
Suppose that the class ${\mathcal M}^{k}_{n}$ is nonempty and contains a manifold $M$. Then $M$ has a poor special spine $P$ with the number of $2$-components equal to $k$ and the number of true vertices equal to $n$. Hence, $\chi(M) = \chi(P) = k-n$. Since the boundary $\partial M$ is totally geodesic and $\chi(M) = \frac{1}{2} \chi(\partial M)$, we have $\chi(M) < 0$. Thus, $k<n$. 
\end{proof}

\begin{remark} {\rm 
By definition, for every $n \geq 2$, the class of manifolds  ${\mathcal M}^{1}_{n}$  coincides with the class of manifolds  ${\mathcal M}_{n}$ from Example~\ref{example1}. Hence, if $M \in {\mathcal M}^{1}_{n}$, then $c(M) = n$. }
\end{remark}

The following theorem establishes exact values of complexity for manifolds in the classes ${\mathcal M}^{2}_{n}$.  

\begin{theorem} \label{theorem5}
If $M \in {\mathcal M}^{2}_{n}$, $n \geq 3$, then $c(M)=n$.
\end{theorem}

\begin{proof}
Since $M \in {\mathcal M}^{2}_{n}$, there exists a poor special spine  $P$ of $M$ with two $2$-components and $n$ true vertices. Thus, $\chi(M) = \chi(P) = 2-n$. By Theorem~3, the manifold $M$ has a special spine $Q$ with $c(M)$ true vertices. Denote by $m \geq 1$ the number $2$-components of $Q$. Then $\chi(M) = \chi (Q) = m - c(M)$. Comparing the expressions for  $\chi(M)$, we obtain  
\begin{equation}
m - 2 = c(M) - n . \label{eq-m}
\end{equation}
Since $c(M) \leq n$, it follows that $m \leq 2$. Let us verify $m \neq 1$. Indeed, if $m=1$, then, by Lemma~1, the special spine $Q$ is poor. In this case, it follows from equality~(\ref{eq-m}) that $c(M) = n-1$, i.e., the poor special spines  $Q$ and $P$ have different numbers of true vertices. This contradicts Proposition~4. 

Thus, $m=2$. Hence, $c(M)= n$. 
 \end{proof}

\begin{remark} {\rm 
Recall the definition of a class of manifolds ${\mathcal H}_{p,q}$ introduced in~\cite{8}. Namely, a connected orientable hyperbolic $3$-manifold $M$ with totally geodesic boundary belongs to the class ${\mathcal H}_{p,q}$ if the following conditions are satisfied:
\begin{itemize}
\item[(1)] the boundary $\partial M$ consists of $q$ components; 
\item[(2)] $M$ has a special spine with the number of $2$-components equal to $p$, and $M$ has no special spines with a smaller number of $2$-components.  
\end{itemize} 
The definition of a poor special spine and Theorem~4 imply that ${\mathcal M}^{2}_{n} \subset {\mathcal H}_{2,1}$. }
\end{remark}

\section{Examples of manifolds in the classes ${\mathcal M}^{2}_{n}$} \label{sec6} 

By Proposition~5, the classes ${\mathcal M}^{2}_{1}$ and ${\mathcal M}^{2}_{2}$ are empty. The classes ${\mathcal M}^{2}_{3}$ and ${\mathcal M}^{2}_{4}$ are not empty. Indeed, Table~1 shows that the class ${\mathcal M}^{2}_{3}$ contains at least $45$ manifolds, and Table~2 shows that the class ${\mathcal M}^{2}_{4}$ contains at least $1421$ manifolds. Below we show that the classes ${\mathcal M}^{2}_{n}$ are nonempty for infinitely many values of $n$. 
  
Let $s\geq 0$ be an integer and $n = 5+4s$. We construct a plane $4$-regular graph $G_{n}$ with decoration of vertices and edges as follows. The graph $G_{n}$ has $n$ vertices, two loops, and $n-1$ double edges. At each vertex of the graph, over- and under-passing are specified (just as at a double point in a knot diagram), and each edge is assigned an element of the cyclic group $\mathbb Z_{3} = \{ 0,1,2 \}$. The decorated graph $G_{5}$ is shown in Fig.~\ref{figG5}. 
\begin{figure}[h]
\begin{center}
\special{em:linewidth 1.4pt} \unitlength=.45mm
\begin{picture}(0,50)(0,0)
\thicklines
\put(-80,20){\circle*{3}}
\put(-40,20){\circle*{3}}
\put(0,20){\circle*{3}}
\put(40,20){\circle*{3}}
\put(80,20){\circle*{3}}
\qbezier(-80,20)(-60,40)(-43,23)
\qbezier(-77,17)(-60,0)(-40,20)
\qbezier(-40,20)(-20,40)(-3,23)
\qbezier(-37,17)(-20,0)(0,20)
\qbezier(0,20)(20,40)(37,23)
\qbezier(3,17)(20,0)(40,20)
\qbezier(40,20)(60,40)(77,23)
\qbezier(43,17)(60,0)(80,20)
\qbezier(-80,20)(-80,10)(-90,10)
\qbezier(-90,10)(-100,10)(-100,20)
\qbezier(-100,20)(-100,30)(-90,30)
\qbezier(-90,30)(-83,30)(-81,24)
\qbezier(80,20)(80,30)(90,30)
\qbezier(90,30)(100,30)(100,20)
\qbezier(100,20)(100,10)(90,10)
\qbezier(90,10)(83,10)(81,16)
\qbezier(-40,40)(-40,37)(-40,35)
\qbezier(-40,30)(-40,27)(-40,25)
\qbezier(-40,15)(-40,15)(-40,10)
\qbezier(-40,5)(-40,5)(-40,0)
\qbezier(40,40)(40,37)(40,35)
\qbezier(40,30)(40,27)(40,25)
\qbezier(40,15)(40,15)(40,10)
\qbezier(40,5)(40,5)(40,0)
\put(-105,20){\makebox(0,0)[c]{$1$}}
\put(105,20){\makebox(0,0)[c]{$1$}}
\put(-60,35){\makebox(0,0)[c]{$0$}}
\put(-60,5){\makebox(0,0)[c]{$1$}}
\put(-20,35){\makebox(0,0)[c]{$1$}}
\put(-20,5){\makebox(0,0)[c]{$1$}}
\put(20,35){\makebox(0,0)[c]{$0$}}
\put(20,5){\makebox(0,0)[c]{$0$}}
\put(60,35){\makebox(0,0)[c]{$0$}}
\put(60,5){\makebox(0,0)[c]{$1$}}
\end{picture}
\end{center} \caption{Decorated graph $G_{5}$.} \label{figG5}
\end{figure}
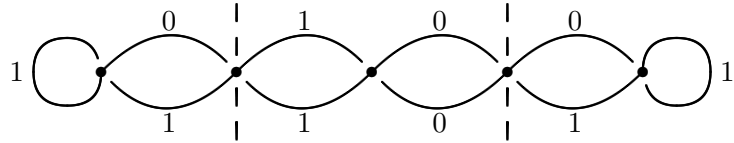
%%%
The graph $G_{5}$ has a block structure: it is composed of three subgraphs $A$, $C$, and $E$ shown in Fig.~\ref{figG5block}. We express this fact as $G_{5} = A \cdot C \cdot E$. Each of graphs $A$ and $E$ has one $4$-valent vertex and one $2$-valent vertex. The graph $C$ has one $4$-valent vertex and two $2$-valent vertices. The decoration of the vertices and edges of the graphs $A$, $C$, and $E$ are induced by the decoration of the graph  $G_{5}$. 
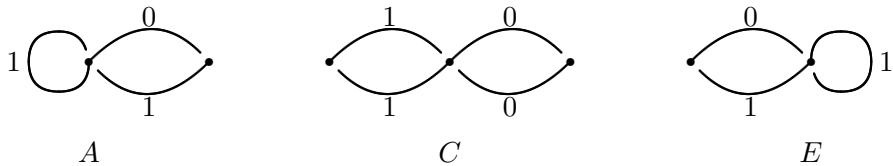
\begin{figure}[h]
\begin{center}
\special{em:linewidth 1.4pt} \unitlength=.4mm
\begin{picture}(0,60)(0,-10)
\thicklines
\put(-40,20){\circle*{3}}
\put(0,20){\circle*{3}}
\put(40,20){\circle*{3}}
\qbezier(-40,20)(-20,40)(-3,23)
\qbezier(-37,17)(-20,0)(0,20)
\qbezier(0,20)(20,40)(37,23)
\qbezier(3,17)(20,0)(40,20)
%left
\put(-80,20){\circle*{3}}
\put(-120,20){\circle*{3}}
\qbezier(-120,20)(-120,10)(-130,10)
\qbezier(-130,10)(-140,10)(-140,20)
\qbezier(-140,20)(-140,30)(-130,30)
\qbezier(-130,30)(-123,30)(-121,24)
\qbezier(-120,20)(-100,40)(-83,23)
\qbezier(-117,17)(-100,0)(-80,20)
\put(-120,-10){\makebox(0,0)[c]{$A$}}
\put(-145,20){\makebox(0,0)[c]{$1$}}
\put(-100,35){\makebox(0,0)[c]{$0$}}
\put(-100,5){\makebox(0,0)[c]{$1$}}
%right
\put(80,20){\circle*{3}}
\put(120,20){\circle*{3}}
\qbezier(120,20)(120,30)(130,30)
\qbezier(130,30)(140,30)(140,20)
\qbezier(140,20)(140,10)(130,10)
\qbezier(130,10)(123,10)(121,16)
\put(145,20){\makebox(0,0)[c]{$1$}}
\put(120,-10){\makebox(0,0)[c]{$E$}}
\put(100,35){\makebox(0,0)[c]{$0$}}
\put(100,5){\makebox(0,0)[c]{$1$}}
\qbezier(80,20)(100,40)(117,23)
\qbezier(83,17)(100,0)(120,20)
%%%
\put(0,-10){\makebox(0,0)[c]{$C$}}
\put(-20,35){\makebox(0,0)[c]{$1$}}
\put(-20,5){\makebox(0,0)[c]{$1$}}
\put(20,35){\makebox(0,0)[c]{$0$}}
\put(20,5){\makebox(0,0)[c]{$0$}}
\end{picture}
\end{center} \caption{Subgraphs $A$, $C$, and $E$ of the graph $G_{5}$.} \label{figG5block}
\end{figure}

Denote graphs $B$ and $D$ as shown in Fig.~\ref{figBDblocks}. The graphs $B$ and $D$ have the same combinatorial structure as the graph $C$ and the same decoration of vertices, however, they differ by the decorations of edges. 
\begin{figure}[h]
\begin{center}
\special{em:linewidth 1.4pt} \unitlength=.4mm
\begin{picture}(0,60)(0,-10)
\thicklines
%left
\put(-120,20){\circle*{3}}
\put(-80,20){\circle*{3}}
\put(-40,20){\circle*{3}}
\qbezier(-120,20)(-100,40)(-83,23)
\qbezier(-117,17)(-100,0)(-80,20)
\qbezier(-80,20)(-60,40)(-43,23)
\qbezier(-77,17)(-60,0)(-40,20)
\put(-80,-10){\makebox(0,0)[c]{$B$}}
\put(-100,35){\makebox(0,0)[c]{$1$}}
\put(-100,5){\makebox(0,0)[c]{$1$}}
\put(-60,35){\makebox(0,0)[c]{$0$}}
\put(-60,5){\makebox(0,0)[c]{$1$}}
%right
\put(40,20){\circle*{3}}
\put(80,20){\circle*{3}}
\put(120,20){\circle*{3}}
\qbezier(40,20)(60,40)(77,23)
\qbezier(43,17)(60,0)(80,20)
\qbezier(80,20)(100,40)(117,23)
\qbezier(83,17)(100,0)(120,20)
\put(80,-10){\makebox(0,0)[c]{$D$}}
\put(60,35){\makebox(0,0)[c]{$0$}}
\put(60,5){\makebox(0,0)[c]{$1$}}
\put(100,35){\makebox(0,0)[c]{$1$}}
\put(100,5){\makebox(0,0)[c]{$0$}} 
\end{picture}
\end{center} \caption{Graphs $B$ and $D$.} \label{figBDblocks}
\end{figure}
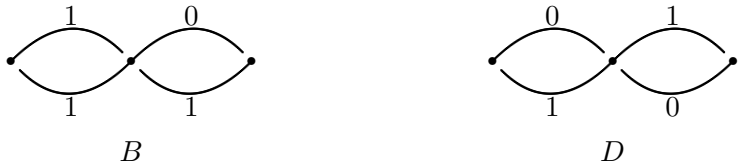
%%%

Define a decorated graph $G_{n}$ as a graph composed successively of the subgraph $A$, $s$ copies of the subgraph $B$, the subgraph $C$, $s$ copies of the subgraph $D$, and the subgraph $E$. In other words, $G_{n} = A \cdot B^{s} \cdot C \cdot D^{s} \cdot E$. The graph $G_{9}$ is demonstrated in Fig.~\ref{figG9}. 
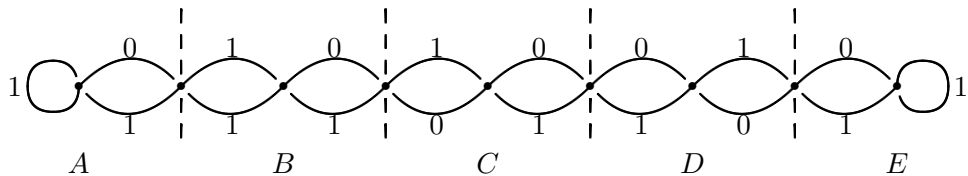
\begin{figure}[h]
\begin{center}
\special{em:linewidth 1.4pt} \unitlength=.34mm
\begin{picture}(0,60)(0,-10)
\thicklines
\put(-160,20){\circle*{3}}
\put(-120,20){\circle*{3}}
\put(-80,20){\circle*{3}}
\put(-40,20){\circle*{3}}
\put(0,20){\circle*{3}}
\put(40,20){\circle*{3}}
\put(80,20){\circle*{3}}
\put(120,20){\circle*{3}}
\put(160,20){\circle*{3}} 
\qbezier(-160,20)(-140,40)(-123,23)
\qbezier(-157,17)(-140,0)(-120,20)
\qbezier(-120,20)(-100,40)(-83,23)
\qbezier(-117,17)(-100,0)(-80,20)
\qbezier(-80,20)(-60,40)(-43,23)
\qbezier(-77,17)(-60,0)(-40,20)
\qbezier(-40,20)(-20,40)(-3,23)
\qbezier(-37,17)(-20,0)(0,20)
\qbezier(0,20)(20,40)(37,23)
\qbezier(3,17)(20,0)(40,20)
\qbezier(40,20)(60,40)(77,23)
\qbezier(43,17)(60,0)(80,20)
\qbezier(80,20)(100,40)(117,23)
\qbezier(83,17)(100,0)(120,20)
\qbezier(120,20)(140,40)(157,23)
\qbezier(123,17)(140,0)(160,20)
%круг
\qbezier(-160,20)(-160,10)(-170,10)
\qbezier(-170,10)(-180,10)(-180,20)
\qbezier(-180,20)(-180,30)(-170,30)
\qbezier(-170,30)(-163,30)(-161,24)
%круг
\qbezier(160,20)(160,30)(170,30)
\qbezier(170,30)(180,30)(180,20)
\qbezier(180,20)(180,10)(170,10)
\qbezier(170,10)(163,10)(161,16)
%пунктиры
\qbezier(-120,50)(-120,50)(-120,45)
\qbezier(-120,40)(-120,37)(-120,35)
\qbezier(-120,30)(-120,27)(-120,25)
\qbezier(-120,15)(-120,15)(-120,10)
\qbezier(-120,5)(-120,5)(-120,0)
\qbezier(-40,50)(-40,50)(-40,45)
\qbezier(-40,40)(-40,37)(-40,35)
\qbezier(-40,30)(-40,27)(-40,25)
\qbezier(-40,15)(-40,15)(-40,10)
\qbezier(-40,5)(-40,5)(-40,0)
\qbezier(40,50)(40,50)(40,45)
\qbezier(40,40)(40,37)(40,35)
\qbezier(40,30)(40,27)(40,25)
\qbezier(40,15)(40,15)(40,10)
\qbezier(40,5)(40,5)(40,0)
\qbezier(120,50)(120,50)(120,45)
\qbezier(120,40)(120,37)(120,35)
\qbezier(120,30)(120,27)(120,25)
\qbezier(120,15)(120,15)(120,10)
\qbezier(120,5)(120,5)(120,0)
\put(-160,-10){\makebox(0,0)[c]{$A$}}
\put(-80,-10){\makebox(0,0)[c]{$B$}}
\put(0,-10){\makebox(0,0)[c]{$C$}}
\put(80,-10){\makebox(0,0)[c]{$D$}}
\put(160,-10){\makebox(0,0)[c]{$E$}}
\put(-185,20){\makebox(0,0)[c]{$1$}}
\put(185,20){\makebox(0,0)[c]{$1$}}
\put(-140,35){\makebox(0,0)[c]{$0$}}
\put(-140,5){\makebox(0,0)[c]{$1$}}
\put(-100,35){\makebox(0,0)[c]{$1$}}
\put(-100,5){\makebox(0,0)[c]{$1$}}
\put(-60,35){\makebox(0,0)[c]{$0$}}
\put(-60,5){\makebox(0,0)[c]{$1$}}
\put(-20,35){\makebox(0,0)[c]{$1$}}
\put(-20,5){\makebox(0,0)[c]{$0$}}
\put(20,35){\makebox(0,0)[c]{$0$}}
\put(20,5){\makebox(0,0)[c]{$1$}}
\put(60,35){\makebox(0,0)[c]{$0$}}
\put(60,5){\makebox(0,0)[c]{$1$}}
\put(100,35){\makebox(0,0)[c]{$1$}}
\put(100,5){\makebox(0,0)[c]{$0$}}
\put(140,35){\makebox(0,0)[c]{$0$}}
\put(140,5){\makebox(0,0)[c]{$1$}} 
\end{picture}
\end{center} \caption{Decorated graph $G_{9}$.} \label{figG9}
\end{figure}

Note that the graphs $G_{n}$ belong to the class of o-graphs defined in~\cite{23}. In that paper, R.~Benedetti and C.~Petronio presented an algorithm for constructing a special polyhedron from an arbitrary o-graph and proved that such a polyhedron is a special spine of a compact orientable $3$-manifold with nonempty boundary. Let $W_{n}$ and $P_{n}$ be a manifold and its special spine, respectively, that are constructed from the o-graph $G_{n}$ with the use of the algorithm from~\cite{23}.

\begin{theorem} \label{theoremexample}
For every $n = 5 + 4 s$, where $s$ is a nonnegative integer, the manifold $W_{n}$ belongs to the class ${\mathcal M}^{2}_{n}$. 
\end{theorem} 

\begin{proof}
In accordance with the definition of the class ${\mathcal M}^{2}_{n}$, we show that 
\begin{itemize}
\item[(1)]  the special polyhedron $P_{n}$ has $n$ true vertices and two $2$-components and is poor;  
\item[(2)] he manifold $W_{n}$ is a hyperbolic manifold with totally geodesic boundary.
\end{itemize}

Let us prove assertion (1). According to the algorithm from~\cite{23} for constructing a spine by an o-graph, to obtain a spine corresponding to the graph $G_{n}$, one should replace the subgraphs $A$, $B$, $C$, $D$, and $E$ by the similarly named blocks shown in Fig.~\ref{fig101} and~\ref{fig102}. As a result of such a join of blocks, we obtain the spine  $P_{n}$. 
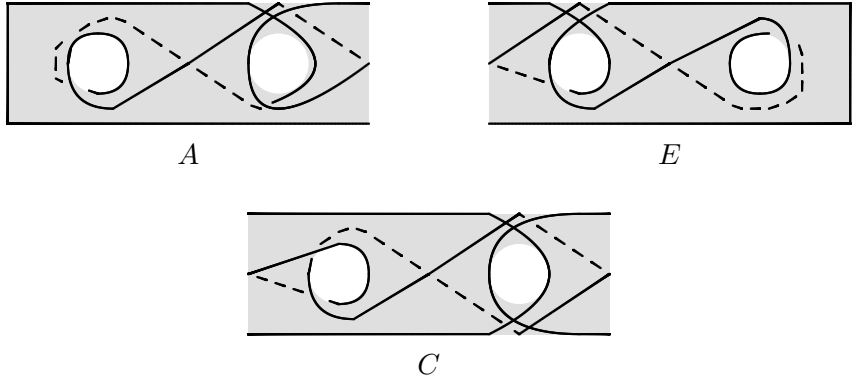
\begin{figure}[h]
\begin{center}
\special{em:linewidth 1.4pt} \unitlength=.4mm
\begin{picture}(0,100)(0,0)
\thicklines
%left
\put(-80,70){\begin{picture}(0,50)(0,-10)
\put(-60,0){\begin{tikzpicture}
  \begin{scope} 
    \fill[gray!25]
    (0mm,0mm) rectangle (24mm,16mm) ;
      \clip (12mm,8mm) circle (4mm);
    \fill[white] (0,0) rectangle (24mm,16mm);
	  \end{scope}
\end{tikzpicture}
}
\put(0,0){\begin{tikzpicture}
  \begin{scope} 
    \fill[gray!25]
    (0mm,0mm) rectangle (24mm,16mm) ;
      \clip (12mm,8mm) circle (4mm);
    \fill[white] (0,0) rectangle (24mm,16mm);
	  \end{scope}
\end{tikzpicture}
}
\qbezier(-60,0)(-60,0)(60,0)
\qbezier(-60,40)(-60,40)(20,40)
\qbezier(28,7)(60,20)(20,40)
\qbezier(-60,0)(-60,0)(-60,40)
\qbezier(-40,20)(-40,30)(-30,30)
\qbezier(-30,30)(-20,30)(-20,20)
\qbezier(-20,20)(-20,10)(-30,10)
\qbezier(-30,10)(-30,10)(-33,11)
\qbezier(-42,14)(-42,14)(-44,15)
\qbezier(-44,18)(-44,18)(-44,22)
\qbezier(-44,25)(-44,25)(-42,27)
\qbezier(30,5)(40,5)(60,20)
\qbezier(30,5)(20,5)(20,20)
\qbezier(50,40)(20,40)(20,20)
\qbezier(60,40)(60,40)(50,40)
\qbezier(24,5)(24,5)(21,6)
\qbezier(18,8)(18,8)(15,10)
\qbezier(12,12)(12,12)(9,14)
\qbezier(6,16)(6,16)(3,18)
\qbezier(30,40)(30,40)(31.5,39)
\qbezier(36,36)(36,36)(39,34)
\qbezier(42,32)(42,32)(45,30)
\qbezier(48,28)(48,28)(51,26)
\qbezier(54,24)(54,24)(57,22)
\qbezier(30,40)(30,40)(0,20)
\qbezier(0,20)(0,20)(-25,5)
\qbezier(-40,20)(-40,5)(-25,5)
\qbezier(-40,20)(-40,20)(-39,25)
\qbezier(-3,22)(-3,22)(-6,24)
\qbezier(-9,26)(-9,26)(-12,28)
\qbezier(-15,30)(-15,30)(-18,32)
\qbezier(-21,34)(-21,34)(-24,35)
\qbezier(-27,35)(-27,35)(-30,34)
\qbezier(-33,33)(-33,33)(-36,31)
\put(0,-10){\makebox(0,0)[c]{$A$}}
\end{picture}}
%center
\put(0,0){\begin{picture}(0,50)(0,-10)
\put(-60,0){\begin{tikzpicture}
  \begin{scope} 
    \fill[gray!25]
    (0mm,0mm) rectangle (24mm,16mm) ;
      \clip (12mm,8mm) circle (4mm);
    \fill[white] (0,0) rectangle (24mm,16mm);
	  \end{scope}
\end{tikzpicture}
}
\put(0,0){\begin{tikzpicture}
  \begin{scope} 
    \fill[gray!25]
    (0mm,0mm) rectangle (24mm,16mm) ;
      \clip (12mm,8mm) circle (4mm);
    \fill[white] (0,0) rectangle (24mm,16mm);
	  \end{scope}
\end{tikzpicture}
}
\qbezier(-60,0)(-60,0)(20,0)
\qbezier(-60,40)(-60,40)(20,40)
\qbezier(20,0)(60,20)(20,40)
\qbezier(-60,20)(-60,20)(-30,30)
\qbezier(-30,30)(-20,30)(-20,20)
\qbezier(-20,20)(-20,10)(-30,10)
\qbezier(-30,10)(-30,10)(-33,11)
\qbezier(-42,14)(-42,14)(-45,15)
\qbezier(-48,16)(-48,16)(-51,17)
\qbezier(-54,18)(-54,18)(-57,19)
\qbezier(60,0)(60,0)(50,0)
\qbezier(50,0)(20,0)(20,20)
\qbezier(50,40)(20,40)(20,20)
\qbezier(60,40)(60,40)(50,40)
\qbezier(60,20)(60,20)(30,0)
\qbezier(30,0)(30,0)(28.5,1)
\qbezier(24,4)(24,4)(21,6)
\qbezier(18,8)(18,8)(15,10)
\qbezier(12,12)(12,12)(9,14)
\qbezier(6,16)(6,16)(3,18)
\qbezier(30,40)(30,40)(31.5,39)
\qbezier(36,36)(36,36)(39,34)
\qbezier(42,32)(42,32)(45,30)
\qbezier(48,28)(48,28)(51,26)
\qbezier(54,24)(54,24)(57,22)
\qbezier(30,40)(30,40)(0,20)
\qbezier(0,20)(0,20)(-25,5)
\qbezier(-40,20)(-40,5)(-25,5)
\qbezier(-40,20)(-40,20)(-39,25)
\qbezier(-3,22)(-3,22)(-6,24)
\qbezier(-9,26)(-9,26)(-12,28)
\qbezier(-15,30)(-15,30)(-18,32)
\qbezier(-21,34)(-21,34)(-24,35)
\qbezier(-27,35)(-27,35)(-30,34)
\qbezier(-33,32)(-33,32)(-36,30)
\put(0,-10){\makebox(0,0)[c]{$C$}}
\end{picture}}
%%% right
\put(80,70){\begin{picture}(0,50)(0,-10)
\thicklines
\put(-60,0){\begin{tikzpicture}
  \begin{scope} 
    \fill[gray!25]
    (0mm,0mm) rectangle (24mm,16mm) ;
      \clip (12mm,8mm) circle (4mm);
    \fill[white] (0,0) rectangle (24mm,16mm);
	  \end{scope}
\end{tikzpicture}
}
\put(0,0){\begin{tikzpicture}
  \begin{scope} 
    \fill[gray!25]
    (0mm,0mm) rectangle (24mm,16mm) ;
      \clip (12mm,8mm) circle (4mm);
    \fill[white] (0,0) rectangle (24mm,16mm);
	  \end{scope}
\end{tikzpicture}
}
\qbezier(-60,40)(-60,40)(-40,40)
\qbezier(-40,40)(-20,30)(-20,20)
\qbezier(-20,20)(-20,10)(-30,10)
\qbezier(-30,10)(-30,10)(-33,11)
\qbezier(-42,14)(-42,14)(-45,15)
\qbezier(-48,16)(-48,16)(-51,17)
\qbezier(-54,18)(-54,18)(-57,19)
\qbezier(60,40)(60,40)(60,0)
\qbezier(60,0)(60,0)(-60,0)
\qbezier(33,30)(20,30)(20,20)
\qbezier(20,20)(20,10)(30,10)
\qbezier(30,10)(40,10)(40,20)
\qbezier(40,20)(40,35)(30,35)
\qbezier(44,23)(44,23)(43,26)
\qbezier(44,17)(44,17)(44,20)
\qbezier(44,14)(44,14)(44,11)
\qbezier(42,9)(42,9)(39,7)
\qbezier(36,6)(36,6)(33,5)
\qbezier(30,5)(30,5)(27,5)
\qbezier(24,5)(24,5)(21,6)
\qbezier(18,8)(18,8)(15,10)
\qbezier(12,12)(12,12)(9,14)
\qbezier(6,16)(6,16)(3,18)
\qbezier(30,35)(30,35)(0,20)
\qbezier(0,20)(0,20)(-25,5)
\qbezier(-40,20)(-40,5)(-25,5)
\qbezier(-40,20)(-40,20)(-39,25)
\qbezier(-3,22)(-3,22)(-6,24)
\qbezier(-9,26)(-9,26)(-12,28)
\qbezier(-15,30)(-15,30)(-18,32)
\qbezier(-21,34)(-21,34)(-24,36)
\qbezier(-27,38)(-27,38)(-30,40)
\qbezier(-30,40)(-30,40)(-60,20)
\qbezier(-40,20)(-40,30)(-20,40)
\qbezier(-20,40)(-20,40)(60,40)
\put(0,-10){\makebox(0,0)[c]{$E$}}
\end{picture}}
\end{picture}
\end{center} \caption{Blocks $A$, $C$, and $E$.} \label{fig101}
\end{figure}
 
\begin{figure}[h]
\begin{center}
\special{em:linewidth 1.4pt} \unitlength=.4 mm
\begin{picture}(0,50)(0,0)
\put(-80,0){\begin{picture}(0,50)(0,-10)
\thicklines
\put(-60,0){\begin{tikzpicture}
  \begin{scope} 
    \fill[gray!25]
    (0mm,0mm) rectangle (24mm,16mm) ;
      \clip (12mm,8mm) circle (4mm);
    \fill[white] (0,0) rectangle (24mm,16mm);
	  \end{scope}
\end{tikzpicture}
}
\put(0,0){\begin{tikzpicture}
  \begin{scope} 
    \fill[gray!25]
    (0mm,0mm) rectangle (24mm,16mm) ;
      \clip (12mm,8mm) circle (4mm);
    \fill[white] (0,0) rectangle (24mm,16mm);
	  \end{scope}
\end{tikzpicture}
}
\qbezier(-60,0)(-60,0)(60,0)
\qbezier(-60,40)(-60,40)(20,40)
\qbezier(28,7)(60,20)(20,40)
\qbezier(-60,20)(-60,20)(-30,30)
\qbezier(-30,30)(-20,30)(-20,20)
\qbezier(-20,20)(-20,10)(-30,10)
\qbezier(-30,10)(-30,10)(-33,11)
\qbezier(-42,14)(-42,14)(-45,15)
\qbezier(-48,16)(-48,16)(-51,17)
\qbezier(-54,18)(-54,18)(-57,19)
\qbezier(30,5)(40,5)(60,20)
\qbezier(30,5)(20,5)(20,20)
\qbezier(50,40)(20,40)(20,20)
\qbezier(60,40)(60,40)(50,40)
\qbezier(24,5)(24,5)(21,6)
\qbezier(18,8)(18,8)(15,10)
\qbezier(12,12)(12,12)(9,14)
\qbezier(6,16)(6,16)(3,18)
\qbezier(30,40)(30,40)(31.5,39)
\qbezier(36,36)(36,36)(39,34)
\qbezier(42,32)(42,32)(45,30)
\qbezier(48,28)(48,28)(51,26)
\qbezier(54,24)(54,24)(57,22)
\qbezier(30,40)(30,40)(0,20)
\qbezier(0,20)(0,20)(-25,5)
\qbezier(-40,20)(-40,5)(-25,5)
\qbezier(-40,20)(-40,20)(-39,25)
\qbezier(-3,22)(-3,22)(-6,24)
\qbezier(-9,26)(-9,26)(-12,28)
\qbezier(-15,30)(-15,30)(-18,32)
\qbezier(-21,34)(-21,34)(-24,35)
\qbezier(-27,35)(-27,35)(-30,34)
\qbezier(-33,32)(-33,32)(-36,30)
\end{picture}
\put(0,-10){\makebox(0,0)[c]{$B$}}
}
\put(80,0){\begin{picture}(0,50)(0,-10)
\thicklines
\put(-60,0){\begin{tikzpicture}
  \begin{scope} 
    \fill[gray!25]
    (0mm,0mm) rectangle (24mm,16mm) ;
      \clip (12mm,8mm) circle (4mm);
    \fill[white] (0,0) rectangle (24mm,16mm);
	  \end{scope}
\end{tikzpicture}
}
\put(0,0){\begin{tikzpicture}
  \begin{scope} 
    \fill[gray!25]
    (0mm,0mm) rectangle (24mm,16mm) ;
      \clip (12mm,8mm) circle (4mm);
    \fill[white] (0,0) rectangle (24mm,16mm);
	  \end{scope}
\end{tikzpicture}
}
\qbezier(-60,40)(-60,40)(-40,40)
\qbezier(-40,40)(-20,30)(-20,20)
\qbezier(-20,20)(-20,10)(-30,10)
\qbezier(-30,10)(-30,10)(-33,11)
\qbezier(-42,14)(-42,14)(-45,15)
\qbezier(-48,16)(-48,16)(-51,17)
\qbezier(-54,18)(-54,18)(-57,19)
\qbezier(60,0)(60,0)(50,0)
\qbezier(50,0)(20,0)(20,20)
\qbezier(33,30)(20,30)(20,20)
\qbezier(42,29)(42,29)(44,28)
\qbezier(46,27)(46,27)(48,26)
\qbezier(50,25)(50,25)(52,24)
\qbezier(54,23)(54,23)(56,22)
\qbezier(60,20)(60,20)(30,0)
\qbezier(30,0)(30,0)(28.5,1)
\qbezier(24,4)(24,4)(21,6)
\qbezier(18,8)(18,8)(15,10)
\qbezier(12,12)(12,12)(9,14)
\qbezier(6,16)(6,16)(3,18)
\qbezier(30,35)(30,35)(0,20)
\qbezier(30,35)(55,20)(25,0)
\qbezier(-60,0)(-60,0)(25,0)
\qbezier(0,20)(0,20)(-25,5)
\qbezier(-40,20)(-40,5)(-25,5)
\qbezier(-40,20)(-40,20)(-39,25)
\qbezier(-3,22)(-3,22)(-6,24)
\qbezier(-9,26)(-9,26)(-12,28)
\qbezier(-15,30)(-15,30)(-18,32)
\qbezier(-21,34)(-21,34)(-24,36)
\qbezier(-27,38)(-27,38)(-30,40)
\qbezier(-30,40)(-30,40)(-60,20)
\qbezier(-40,20)(-40,30)(-20,40)
\qbezier(-20,40)(-20,40)(60,40)
\put(0,-10){\makebox(0,0)[c]{$D$}}
\end{picture}}
\end{picture}
\end{center} \caption{Blocks $B$ and $D$.} \label{fig102}
\end{figure}
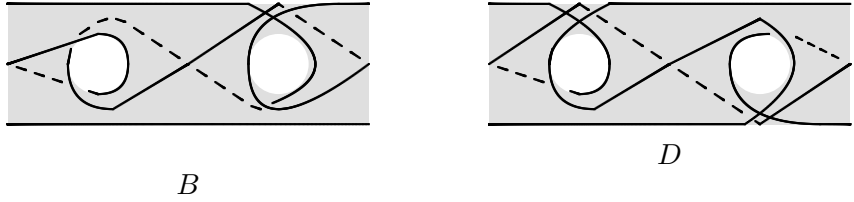

According to the construction algorithm, the number of true vertices of the spine $P_{n}$ is equal to the number of vertices of the o-graph $G_{n}$, i.e., to $n$. The curves shown in the figures of blocks are joined into closed curves. It can be easily verified that for any n the number of closed curves is equal to two. Since these closed curves correspond to $2$-components of the spine, we find that the number of 2-components of the spine $P_{n}$ is equal to two.

Let us show that the special spine $P_{n}$ is poor. If a simple polyhedron  $Q\subset P_{n}$ contains at least one point of a $2$-component $\xi$ of the polyhedron $P_{n}$, then $\xi \subset Q$ by the compactness of simple polyhedra. Hence, to describe a simple subpolyhedron, it suffices to indicate the $2$-components contained in it. Combinatorial analysis shows that none of the two $2$-components of the spine $P_{n}$ defines a simple subpolyhedron.
  
Let us prove assertion (2). Using the duality between special spines of manifolds with boundary and their ideal triangulations, we can assume that $W_{n}$ is glued from $n$ truncated tetrahedra. Denote this triangulation by $\mathcal T$. In the same way as in~\cite{6}, in order to endow $W_{n}$ with a hyperbolic structure, we identify each tetrahedron of the triangulation $\mathcal T$ with a copy of a regular truncated hyperbolic tetrahedron $T^{*}_{2\pi / (3n)}$ with dihedral angles $2\pi / (3n)$. Since all faces of $T^{*}_{2\pi / (3n)}$ are identical symmetric hexagons, each pairwise identification of faces of tetrahedra in $\mathcal T$ can be realized by an isometry that identifies faces of the copies of $T^{*}_{2\pi / (3n)}$. In the construction of the manifold $W_{n}$, all edges of the tetrahedra in ${\mathcal T}$ form two classes of equivalent edges corresponding to two $2$-components of the spine  $P_{n}$. It is easily seen from the construction of the spine $P_{n}$ that the number of edges of the tetrahedra of the triangulation in each equivalence class is $3n$. Since the dihedral angles of the tetrahedron $T^{*}_{2\pi / (3n)}$ are equal to $2\pi / (3n)$, we conclude that $W_{n}$ is a hyperbolic manifold. Since the boundary $\partial W_{n}$ is glued from triangular faces of copies of the truncated hyperbolic tetrahedra $T^{*}_{2\pi / (3n)}$ each of which is orthogonal to the adjacent hexagonal faces, the boundary $\partial W_{n}$ is totally geodesic.
\end{proof}

\begin{remark}{\rm 
Since the manifold $W_{n}$ defined by the spine $P_{n}$ is glued from $n$ regular truncated tetrahedra with dihedral angles $2\pi / (3n)$, its volume is calculated by the formula 
\begin{equation}
\operatorname{vol} (W_{n}) = n \cdot \operatorname{vol} (T^{*}_{2\pi/(3n)}), 
\end{equation} 
where the volume $\operatorname{vol} (T^{*}_{2\pi/(3n)})$ can be found by formula~(\ref{eq1}) as well as by formula~(\ref{eq2}).   }
\end{remark}

\begin{corollary}
The classes ${\mathcal M}^{2}_{n}$ are nonempty for an infinite number of values of $n$. 
\end{corollary}


\begin{thebibliography}{99}

\bibitem{1} 
S. Matveev, Algorithmic Topology and Classification of 3-Manifolds (Springer, Berlin, 2007).

\bibitem{2} 
S.~V.~Matveev,Tabulation of three-dimensional manifolds, Russ. Math. Surv. {\bf 60}, 673--698 (2005).

\bibitem{3} 
R.~Frigerio, B.~Martelli, C.~Petronio, Small hyperbolic 3-manifolds with geodesic boundary, Exp. Math. {\bf 13} (2), 171--184 (2004).

\bibitem{4} 
S.~Matveev, E.~Fominykh, V.~Potapov, V.~Tarkaev, Atlas of 3-manifolds, http://www.matlas.math.csu.ru

\bibitem{5} 
S.~Anisov, Exact values of complexity for an infinite number of 3-manifolds, Moscow Math. J. {\bf 5} (2), 305--310
(2005).

\bibitem{6} 
R.~Frigerio, B.~Martelli, C.~Petronio, Complexity and Heegaard genus of an infinite class of compact 3-manifolds,  Pac. J. Math. {\bf 210} (2), 283--297 (2003).

\bibitem{7} 
A.~Yu.~Vesnin,  E.~A.~Fominykh, Exact values of complexity for Paoluzzi--Zimmermann manifolds, Dokl. Math. {\bf 84} (1), 542--544 (2011). 

\bibitem{8} 
A.~Yu.~Vesnin,  E.~A.~Fominykh, On complexity of three-dimensional hyperbolic ma\-ni\-folds with geodesic boundary, Sib. Math. J. {\bf 53}, 625--634 (2012). 

\bibitem{9} 
W.~Jaco, H.~Rubinstein, S.~Tillmann, Minimal triangulations for an infinite family of lens spaces, J. Topol.
{\bf 2} (1), 157--180 (2009).

\bibitem{10} 
W.~Jaco, J.~H.~Rubinstein, S.~Tillmann, Coverings and minimal triangulations of 3-manifolds, Algebr. Geom. Topol. {\bf 11} (3), 1257--1265 (2011).

\bibitem{11} 
A.~Yu.~Vesnin, V.~V.~Tarkaev, E.~A.~Fominykh, On the complexity of three-dimensional cusped hyperbolic
manifolds, Dokl. Math. {\bf 89} (3), 267--270 (2014).

\bibitem{12} 
A.~Yu.~Vesnin, V.~V.~Tarkaev, E.~A.~Fominykh, Three-dimensional cusped hyperbolic manifolds of complexity 10 with maximal volume, Proc. Steklov Inst. Math. {\bf 289} (Suppl. 1) (2015). 

\bibitem{13} 
S.~Matveev, C.~Petronio, A.~Vesnin, Two-sided asymptotic bounds for the complexity of some closed hyperbolic three-manifolds, J. Aust. Math. Soc. {\bf 86} (2), 205--219 (2009).

\bibitem{14} 
C.~Petronio,  A.~Vesnin, Two-sided bounds for the complexity of cyclic branched coverings of two-bridge links,  Osaka J. Math. {\bf 46} (4), 1077--1095 (2009).

\bibitem{15} 
E.~A.~Fominykh, Dehn surgeries on the figure eight knot: An upper bound for complexity, Sib. Math. J. {\bf 52}, 537--543 (2011).

\bibitem{16} 
E.~Fominykh,  B.~Wiest, Upper bounds for the complexity of torus knot complements, J. Knot Theory Ramif. {\bf 22} (10), 1350053 (2013).

\bibitem{17} 
A.~Yu.~Vesnin,  E.~A.~Fominykh, Two-sided bounds for the complexity of hyperbolic three-manifolds with geodesic boundary,  Proc. Steklov Inst. Math. {\bf 286}, 55--64 (2014). 

\bibitem{18} 
B.~G.~Casler, An imbedding theorem for connected 3-manifolds with boundary, Proc. Am. Math. Soc. {\bf 16}, 559--566 (1965).

\bibitem{19}  
Y.~Miyamoto, Volumes of hyperbolic manifolds with geodesic boundary, Topology {\bf 33} (4), 613--629 (1994).

\bibitem{20}  
M.~Fujii, Hyperbolic 3-manifolds with totally geodesic boundary which are decomposed into hyperbolic truncated
tetrahedra, Tokyo J. Math. {\bf 13} (2), 353--373 (1990).

\bibitem{24}
Homepage of Prof. Carlo Petronio: http://www.dm.unipi.it/pages/petronio. 

\bibitem{21} 
E.~Fominykh,  B.~Martelli, k-Normal surfaces, J. Diff. Geom. {\bf 82} (1), 101--114 (2009). 

\bibitem{22}  
V.~G.~Turaev,  O.~Y.~Viro, State sum invariants of 3-manifolds and quantum 6j-symbols, Topology {\bf 31} (4), 865--902 (1992).

\bibitem{23}
R.~Benedetti, C.~Petronio, A finite graphic calculus for 3-manifolds, Manuscr. Math. {\bf 88}, 291--310 (1995).

\end{thebibliography}
\end{document}